\documentclass[11pt]{amsart}
\usepackage[dvipsnames]{xcolor}
\usepackage[utf8]{inputenc}
\usepackage{amssymb,amsmath,amsthm,amsfonts,fullpage,float,latexsym,bbm,microtype,cite,fancyvrb}
\usepackage{tikz}
\usepackage{enumerate}
\usetikzlibrary{decorations.pathreplacing}
\usepackage{hyperref}
\hypersetup{colorlinks=true, linkcolor=blue, citecolor=magenta, filecolor=magenta, urlcolor=magenta}
\usepackage{bm}
\usepackage{diagbox}

\makeatletter
\newtheorem*{rep@theorem}{\rep@title}\newcommand{\newreptheorem}[2]{%
\newenvironment{rep#1}[1]{%
\def\rep@title{\bf #2 \ref{##1}}%
\begin{rep@theorem}}%
{\end{rep@theorem}}}
\makeatother
\newreptheorem{theorem}{Theorem}

\newtheorem{theorem}{Theorem}[section]
\newtheorem{proposition}[theorem]{Proposition}

\newtheorem{corollary}[theorem]{Corollary}
\theoremstyle{definition}
\newtheorem{remark}[theorem]{Remark}

\renewcommand{\sl}{\mathfrak{sl}}

\newcommand{\Q}{\mathbb{Q}}
\newcommand{\R}{\mathbb{R}}
\newcommand{\Z}{\mathbb{Z}}
\newcommand{\C}{\mathbb{C}}

\DeclareMathOperator{\Hilb}{Hilb}

\DeclareMathOperator{\Char}{Char}
\DeclareMathOperator{\GL}{GL}
\DeclareMathOperator{\IrrChar}{IrrChar}
\DeclareMathOperator{\Cat}{Cat}

\DeclareMathOperator{\Sym}{Sym}

\begin{document}

\title{Universal series for dihedral group coinvariant rings}

\author[J. Lentfer]{John Lentfer}
\address{Department of Mathematics\\
         University of California, Berkeley, CA, USA}
\email{jlentfer@berkeley.edu}


\begin{abstract}
In 1994, Alfano determined a monomial basis, bigraded Hilbert series, and bigraded Frobenius series for the ring of diagonal dihedral group coinvariants $R^{(2,0)}_{\mathfrak{I}_{2}(n)}$. 
Using diagonal supersymmetry, we determine a universal multigraded character series, universal multigraded Hilbert series, and monomial basis for the generalization to any $k$ sets of bosonic variables and $j$ sets of fermionic variables $R^{(k,j)}_{\mathfrak{I}_{2}(n)}$.
\end{abstract}

\maketitle

\section{Introduction}\label{sec:introduction}

The diagonal coinvariant ring $R_n^{(2,0)} = \C[\bm{x},\bm{y}]/\allowbreak \langle \C[\bm{x},\bm{y}]_+^{\mathfrak{S}_n} \rangle$ was studied by Haiman in 1994 \cite{Haiman1994}, and has been of much interest in algebraic combinatorics since. Here $\bm{x}$ and $\bm{y}$ respectively denote two sets of variables $\{x_1,\ldots,x_n\}$ and $\{y_1,\ldots,y_n\}$. The defining ideal $\langle \C[\bm{x},\bm{y}]_+^{\mathfrak{S}_n} \rangle$ is generated by all polynomials in $\C[\bm{x},\bm{y}]$ which are invariant under the diagonal action of $\mathfrak{S}_n$, with no constant term. In 1994, Haiman conjectured the dimension, bigraded Hilbert series, and Frobenius series of $R_n^{(2,0)}$, however these resisted proof until 2002, where he proved them using new results on the Hilbert scheme of points in the complex plane \cite{Haiman2002}.

In the interim, Alfano \cite{Alfano} studied the related ring $R^{(2,0)}_{\mathfrak{I}_{2}(n)} = \C[x_1, x_2, y_1, y_2]/\allowbreak \langle \C[x_1, x_2, y_1, y_2]_+^{\mathfrak{I}_2(n)} \rangle$, where $\mathfrak{I}_{2}(n)$ denotes the dihedral group of order $2n$. 
This ring is easier to study than $R_n^{(2,0)}$ since every dihedral group is of rank 2, while the rank of the symmetric group $\mathfrak{S}_n$ grows with $n$. Alfano found the dimension, bigraded Hilbert series, and bigraded character series of $R^{(2,0)}_{\mathfrak{I}_{2}(n)}$ using combinatorial constructions, linear algebra, and Lie algebra operators, without any new results in algebraic geometry. 

There has been recent interest (see \cite{BergeronOPAC, Bergeron2020, BHIR, Zabrocki2020}) in considering coinvariant rings with $k$ sets of commuting (bosonic) variables $\bm{x}^{(1)}, \ldots, \bm{x}^{(k)}$ and $j$ sets of anticommuting (fermionic) variables $\bm{\theta}^{(1)}, \ldots, \bm{\theta}^{(j)}$.
We denote this by 
\begin{equation} R_G^{(k,j)} = \C[\bm{x}^{(1)}, \ldots, \bm{x}^{(k)}, \bm{\theta}^{(1)}, \ldots, \bm{\theta}^{(j)}]/\langle \C[\bm{x}^{(1)}, \ldots, \bm{x}^{(k)}, \bm{\theta}^{(1)}, \ldots, \bm{\theta}^{(j)}]_+^{G} \rangle,\end{equation}
where $G \subseteq \GL(n)$ is a finite group acting diagonally on all sets of variables. In the case where $G = \mathfrak{S}_n$ acts by permuting variables, we write $R_n^{(k,j)} := R_{\mathfrak{S}_n}^{(k,j)}$.
Commuting variables commute with all variables. Anticommuting variables anticommute with all anticommuting variables, that is, $\theta_i^{(\ell)} \theta_j^{(m)} = - \theta_j^{(m)} \theta_i^{(\ell)}$, which implies that $(\theta_i^{(\ell)})^2 = 0$.
Note that $\C[\bm{x}^{(1)}, \ldots, \bm{x}^{(k)}, \bm{\theta}^{(1)}, \ldots, \bm{\theta}^{(j)}]$ comes from taking coordinates on $(\Sym \C^n)^{\otimes k} \otimes (\wedge \C^n)^{\otimes j}$.

In \cite{Lentfer-Supersymmetry}, it was shown that $R_G^{(k,j)}$ is a $\mathcal{U}(\mathfrak{gl}(k|j)) \otimes \C[G]$-module, where $\mathcal{U}(\mathfrak{gl}(k|j))$ is the universal enveloping algebra of the Lie superalgebra $\mathfrak{gl}(k|j)$. 
As a consequence, for fixed $n$ and $G \subseteq \GL(n)$, the multigraded character series of $R_{G}^{(k,j)}$ can be written in terms of universal coefficients $c_{\lambda,\mu}$ and super Schur functions as
\begin{equation}\label{eq:diagonal-supersymmetry}
    \Char(R_G^{(k,j)}; \mathbf{q};\mathbf{u}) =  \sum_{\lambda \in P(k,j,n)} \sum_{\chi^\mu \in \IrrChar(G)}c_{\lambda,\mu} s_\lambda(\mathbf{q}/\mathbf{u})\chi^\mu,
\end{equation}
where $P(k,j,n)$ denotes the set of partitions $\lambda$ with length $\ell(\lambda) \leq n$ which satisfy $\lambda_{k+1} \leq j$, and $\IrrChar(G)$ denotes the set of irreducible characters $\chi^\mu$ of $G$.
This is a general, but abstract result, since there is no known explicit formula for all of these coefficients.
In the important case of $G = \mathfrak{S}_n$ acting by permuting variables, the universal series is currently known explicitly only for $n \leq 6$ \cite[Section 2.3]{BergeronOPAC}.

The goal of this paper is to give explicit universal series coefficients and an explicit monomial basis when $G = \mathfrak{I}_{2}(n)$. Our main object of study is 
\begin{align}R^{(k,j)}_{\mathfrak{I}_{2}(n)} =&\ \C[x_1^{(1)}, x_2^{(1)}, x_1^{(2)}, x_2^{(2)}, \ldots, x_1^{(k)}, x_2^{(k)}, \theta_1^{(1)}, \theta_2^{(1)}, \theta_1^{(2)}, \theta_2^{(2)}, \ldots, \theta_1^{(j)}, \theta_2^{(j)}]\\
&/\langle \C[x_1^{(1)}, x_2^{(1)}, x_1^{(2)}, x_2^{(2)}, \ldots, x_1^{(k)}, x_2^{(k)}, \theta_1^{(1)}, \theta_2^{(1)}, \theta_1^{(2)}, \theta_2^{(2)}, \ldots, \theta_1^{(j)}, \theta_2^{(j)}]_+^{\mathfrak{I}_{2}(n)} \rangle,\nonumber\end{align}
where, as before, the defining ideal is generated by all polynomials which are invariant under the diagonal action of $\mathfrak{I}_{2}(n)$, with no constant term.

In Section~\ref{sec:background-preliminaries}, we review background. 
In the next three sections, we study $R^{(k,j)}_{\mathfrak{I}_{2}(n)}$ for arbitrary non-negative integers $(k,j)$, and prove our main results.
In Section~\ref{sec:character-series}, using diagonal supersymmetry and a theorem of Alfano, we derive an explicit multigraded character series for $R^{(k,j)}_{\mathfrak{I}_{2}(n)}$.
\begin{theorem}\label{thm:b-f-character}
Let $\mathbf{q} = (q_1,\ldots, q_k)$ and $\mathbf{u} = (u_1,\ldots, u_j)$. The $\mathbf{q},\mathbf{u}$-graded character series for $R^{(k,j)}_{\mathfrak{I}_{2}(n)}$ is given by
\begin{equation}
    \begin{aligned}
        \Char(R_{\mathfrak{I}_{2}(n)}^{(k,j)}; \mathbf{q};\mathbf{u}) = &\ \chi_1 + s_{(1,1)}(\mathbf{q}/\mathbf{u})\chi_2 + s_{(n)}(\mathbf{q}/\mathbf{u})\chi_2+ \sum_{i=1}^{\lfloor \frac{n-1}{2} \rfloor} (s_{(i)}(\mathbf{q}/\mathbf{u}) + s_{(n-i)}(\mathbf{q}/\mathbf{u})) \chi^i\\
        &+ \begin{cases}
            s_{(\frac{n}{2})}(\mathbf{q}/\mathbf{u})(\chi_3+\chi_4)& \text{ if $n$ even,}\\
        0& \text{ if $n$ odd.}
        \end{cases}
    \end{aligned}
\end{equation}
\end{theorem}
In Section~\ref{sec:hilbert-series}, as a consequence, we obtain an explicit multigraded Hilbert series for $R^{(k,j)}_{\mathfrak{I}_{2}(n)}$ (Theorem~\ref{thm:b-f-hilbert}), and determine the dimension of $R^{(k,j)}_{\mathfrak{I}_{2}(n)}$ (Corollary~\ref{cor:b-f-cardinality-second}).
In Section~\ref{sec:monomial-basis}, we derive a monomial basis for $R^{(k,j)}_{\mathfrak{I}_{2}(n)}$ (Theorem~\ref{thm:b-f-basis}) using straightening relations.

Section~\ref{sec:sign-character} briefly remarks on the sign character and the connection with Catalan numbers of type $I$. In Section~\ref{sec:discussion}, we discuss some numerological implications of these results for specializations of $(k,j)$ to certain small values.
Appendix~\ref{sec:cyclic} contains analogous results for the much simpler case of cyclic groups $\Z_n$.

\section{Preliminaries}\label{sec:background-preliminaries}

\subsection{The dihedral group}

The setting of the paper is $G = \mathfrak{I}_2(n)$, the dihedral group of order $2n$, and we require that $n \geq 2$.\footnote{Sometimes the case where $n=2$ is treated separately from $n > 2$ (for example, $\mathfrak{I}_2(2) \cong \mathfrak{S}_2 \times \mathfrak{S}_2$ is reducible).} 
$\mathfrak{I}_{2}(n)$ is the symmetry group of a regular $n$-gon in the real plane $\R^2$, and is generated by the Euclidean transformations $\rho$, rotation by $2\pi /n$ radians counterclockwise, and $\phi$, reflection about the $x$-axis. 
Matrices for these are given by:
\begin{equation}\label{eq:unitary_matrices}
    \rho = \begin{bmatrix}
        \cos \frac{2\pi}{n} & -\sin\frac{2\pi}{n} \\
        \sin \frac{2\pi}{n} & \cos\frac{2\pi}{n}
    \end{bmatrix} \text{ and }
    \phi = \begin{bmatrix}
        1 & 0 \\
        0 & -1
    \end{bmatrix}.
\end{equation}
Since these matrices are unitary matrices, we may take $\{x_{1}^{(1)},x_{2}^{(1)}\}$ to an orthonormal basis for $\R^2$. 
The action of $\mathfrak{I}_{2}(n)$ can be extended homomorphically to all $\R[\bm{x}^{(1)}]$. 
Since we will use the character theory of $\mathfrak{I}_{2}(n)$, and its irreducible real characters are not necessarily all irreducible over $\C$, we will work with complex characters.
So we complexify $\R^2$ to $\C^2$, and work over $\C$ throughout.
We remark that we could work more minimally over $\Q(\zeta_n)$, where $\zeta_n$ is a primitive $n$th root of unity.

Let $\Re(P)$ denote taking the real part of a complex polynomial $P$.
As our fundamental $\mathfrak{I}_2(n)$-invariants, we use $(x_1^{(1)})^2+(x_2^{(1)})^2$ and from \cite[equation (4.6)]{Alfano},
\begin{equation}\label{eq:real-expansion}
\Re(x_1^{(1)} + i x_2^{(1)})^n = \sum_{m=0}^{\lfloor \frac{n}{2} \rfloor} \binom{n}{2m} (-1)^m (x_{1}^{(1)})^{n-2m}(x_{2}^{(1)})^{2m},
\end{equation}
where here $ i =\sqrt{-1}$.

\subsection{Dihedral group coinvariants and harmonics}
Fix $V = \C^2$.
Note that $\mathfrak{I}_2(n) \subseteq \GL(V) = \GL(2)$. 
Let $\Sym V$ denote the symmetric algebra on $V$ over $\C$ and let $\wedge V$ denote the exterior algebra on $V$ over $\C$.

Here and throughout, let 
\begin{equation}\C[\bm{x}^{(1)}, \ldots, \bm{x}^{(k)}, \bm{\theta}^{(1)}, \ldots, \bm{\theta}^{(j)}]\end{equation} 
denote the polynomial ring $\C[x_1^{(1)}, x_2^{(1)}, \ldots, x_1^{(k)}, x_2^{(k)}, \theta_1^{(1)}, \theta_2^{(1)}, \ldots, \theta_1^{(j)}, \theta_2^{(j)}]$.
The diagonal action (in $k$ sets of bosonic variables and $j$ sets of fermionic variables) of $\mathfrak{I}_{2}(n)$ on $\C[\bm{x}^{(1)}, \ldots, \bm{x}^{(k)},\bm{\theta}^{(1)}, \ldots, \bm{\theta}^{(j)}]$ is defined for any $g \in \mathfrak{I}_{2}(n)$ and polynomial $P(x_1^{(1)}, x_2^{(1)}, \ldots, x_1^{(k)}, x_2^{(k)}, \theta_1^{(1)}, \theta_2^{(1)}, \ldots, \theta_1^{(j)}, \theta_2^{(j)}) \in \C[\bm{x}^{(1)}, \ldots, \bm{x}^{(k)}, \bm{\theta}^{(1)}, \ldots, \bm{\theta}^{(j)}]$ by
\begin{equation}
\begin{aligned} &g \cdot P(x_1^{(1)}, x_2^{(1)}, \ldots, x_1^{(k)}, x_2^{(k)}, \theta_1^{(1)}, \theta_2^{(1)}, \ldots, \theta_1^{(j)}, \theta_2^{(j)})\\ 
&= P(g \cdot x_1^{(1)}, g \cdot x_2^{(1)}, \ldots, g \cdot x_1^{(k)}, g \cdot x_2^{(k)}, g \cdot \theta_1^{(1)}, g \cdot \theta_2^{(1)}, \ldots, g \cdot \theta_1^{(j)}, g \cdot \theta_2^{(j)}).
\end{aligned}
\end{equation}
Note that matrices for this action are generated by block diagonal matrices with $k+j$ blocks of the matrices for $\rho$ and $\phi$ given in equation~(\ref{eq:unitary_matrices}).

Using a certain form where multiplication is adjoint to differentiation, one can define a space of diagonal harmonics in $k$ sets of bosonic variables and $j$ sets of fermionic variables (see \cite[Section 2]{SwansonWallach1} and \cite[Section 2]{SwansonWallach2}).
This harmonic space is isomorphic to the coinvariant ring $R_{\mathfrak{I}_2(n)}^{(k,j)}$ as a multigraded $\mathfrak{I}_2(n)$-module.

\subsection{Dihedral group characters}

We recall the dihedral group irreducible characters (see for example \cite[Chapter 18.3]{JamesLiebeck}). 
We use the group presentation $\mathfrak{I}_{2}(n) = \langle \rho, \phi \ |\ \rho^n = \phi^2 = 1, \phi^{-1}\rho\phi = \rho^{-1}\rangle$. 
The 1-dimensional irreducible characters are:
\begin{itemize}
    \item $\chi_1$ defined by $\chi_1(\rho) = 1$ and $\chi_1(\phi) = 1$,
    \item $\chi_2$ defined by $\chi_2(\rho) = -1$ and $\chi_2(\phi) = 1$,
\end{itemize}
and if and only if $n$ is even, 
\begin{itemize}
    \item $\chi_3$ defined by $\chi_3(\rho) = 1$ and $\chi_3(\phi) = -1$,
    \item $\chi_4$ defined by $\chi_4(\rho) = -1$ and $\chi_4(\phi) = -1$.
\end{itemize}
The 2-dimensional irreducible characters are, for $1 \leq h \leq \lfloor \frac{n-1}{2} \rfloor$:
\begin{itemize}
    \item $\chi^h$ defined by $\chi^h(\rho) = 2\cos(\frac{2h\pi}{n})$ and $\chi^h(\phi) = 0$.
\end{itemize}
These constitute the complete list of irreducible characters.

\begin{remark}\label{rmk:extra-characters}
The definition of the 2-dimensional characters may be extended to for $1 \leq h \leq n-1$:
\begin{itemize}
    \item $\chi^h$ defined by $\chi^h(\rho) = 2\cos(\frac{2h\pi}{n})$ and $\chi^h(\phi) = 0$.
\end{itemize} 
Then $\chi^i$ and $\chi^{n-i}$ are isomorphic, and when $n$ is even, the character $\chi^{\frac{n}{2}}$ is reducible with $\chi^{\frac{n}{2}} = \chi_3+\chi_4$.
\end{remark}

\subsection{Dihedral group polarization operators}

Define polarization operators for $\ell \in \{1,\ldots, k-1\}$ by
\begin{equation}\label{eq:bosonic-E} E_\ell  := \sum_{i=1}^2 x_{i}^{(\ell+1)} \partial_{x_{i}^{(\ell)}},\end{equation}
for $k$,
\begin{equation}\label{eq:bosonic-fermionic-E} E_k := \sum_{i=1}^2 \theta_{i}^{(1)} \partial_{x_{i}^{(k)}},\end{equation}
and for $\ell \in \{k+1,\ldots, k+j-1\}$,
\begin{equation}\label{eq:fermionic-E} E_\ell := \sum_{i=1}^2 \theta_{i}^{(\ell-k+1)} \partial_{\theta_{i}^{(\ell-k)}},\end{equation}
which are all diagonal invariant under the action of $\mathfrak{I}_{2}(n)$; this is shown for $E_1$ in \cite[Section IV.C]{Alfano}, and the same argument holds for the others. 

The operators in equations~(\ref{eq:bosonic-E}-\ref{eq:fermionic-E}) generate the Lie superalgebra $\mathfrak{sl}(k|j)$ \cite[Section 5.2.1]{ChengWangBook}.

\subsection{Super Schur functions}

A super Schur function $s_\lambda(\mathbf{q}/\mathbf{u})$ is defined by 
\begin{equation}
    s_\lambda(\mathbf{q}/\mathbf{u}) = \sum_{\nu \subseteq \lambda} s_\nu(\mathbf{q})s_{\lambda'/\nu'}(\mathbf{u}),
\end{equation}
where $\lambda'$ denotes the transpose of the partition $\lambda$ (see for example \cite[Sections A.2.2 and 3.2]{ChengWangBook}).

A special case that we use often is
\begin{equation}
    s_{(n)}(\mathbf{q}/\mathbf{u}) = \sum_{i=0}^n s_{(n-i)}(\mathbf{q})s_{(1^i)}(\mathbf{u}),
\end{equation}
where $(0)$ means the empty partition $\varnothing$.

\subsection{Diagonal supersymmetry}

We recall the ``diagonal supersymmetry'' theorem, which was previewed in equation~(\ref{eq:diagonal-supersymmetry}).
\begin{theorem}[\!\!{\cite[Theorem 1.2]{Lentfer-Supersymmetry}}]\label{thm:G-main-theorem}
    Fix a positive integer $n$ and a finite group $G \subset \GL(n)$. 
    For partitions $\lambda$ with $\ell(\lambda) \leq n$, and indices $\mu$ of irreducible $G$-characters, there exist nonnegative integer coefficients $c_{\lambda,\mu}$ such that for any $(k,j)$, the multigraded character series of $R_G^{(k,j)}$ is
\begin{equation}\label{eq:G-theorem}
    \Char(R_G^{(k,j)}; \mathbf{q};\mathbf{u}) =  \sum_{\lambda \in P(k,j,n)} \sum_{\chi^\mu \in \IrrChar(G)}c_{\lambda,\mu} s_\lambda(\mathbf{q}/\mathbf{u})\chi^\mu.
\end{equation}
\end{theorem}

By specializing to dihedral groups, the multigraded character series of $R^{(k,j)}_{\mathfrak{I}_{2}(n)}$ is a sum of super Schur functions times $\mathfrak{I}_2(n)$-characters.
\begin{theorem}\label{thm:G-main-theorem-dihedral}
    Fix an integer $n \geq 2$. 
    For partitions $\lambda$ with $\ell(\lambda) \leq 2$, and irreducible $\mathfrak{I}_2(n)$-characters $\chi^\mu$, there exist nonnegative integer coefficients $c_{\lambda,\mu}$ such that for any $(k,j)$, the multigraded character series of $R_{\mathfrak{I}_2(n)}^{(k,j)}$ is
\begin{equation}\label{eq:diagonal-supersymmetry_I}
    \Char(R_{\mathfrak{I}_2(n)}^{(k,j)}; \mathbf{q};\mathbf{u}) =  \sum_{\lambda \in P(k,j,2)} \sum_{\chi^\mu \in \IrrChar(\mathfrak{I}_2(n))}c_{\lambda,\mu} s_\lambda(\mathbf{q}/\mathbf{u})\chi^\mu.
\end{equation}
\end{theorem}
In particular, the condition $\lambda \in P(k,j,2)$ means that $\ell(\lambda) \leq 2$ and $\lambda_{k+1} \leq j$.

\section{The multigraded character series}\label{sec:character-series}

In this section, we derive a multigraded character series for $R^{(k,j)}_{\mathfrak{I}_{2}(n)}$. 
Alfano \cite[equation (4.18)]{Alfano} recorded the singly graded character of $R^{(1,0)}_{\mathfrak{I}_{2}(n)}$ as the following graded version of the regular representation:
\begin{equation}
    \Char(R_{\mathfrak{I}_{2}(n)}^{(1,0)}; q) = \chi_1 + q^n \chi_2 + \sum_{i=1}^{\lfloor \frac{n-1}{2} \rfloor} (q^i + q^{n-i})\chi^i + \begin{cases}
        q^{\frac{n}{2}}(\chi_3 + \chi_4)& \text{ if $n$ even,}\\
        0& \text{ if $n$ odd.}
    \end{cases}
\end{equation}

Alfano \cite[Chapter IV.C]{Alfano} applied the polarization operators $E_1^0, E_1^1, \ldots, E_1^i$ to any bigraded component $(R^{(2,0)}_{\mathfrak{I}_{2}(n)})_{i,0}$ to generate an $\sl(2)$-string of modules $(R^{(2,0)}_{\mathfrak{I}_{2}(n)})_{i,0}, (R^{(2,0)}_{\mathfrak{I}_{2}(n)})_{i-1,1}, \ldots, (R^{(2,0)}_{\mathfrak{I}_{2}(n)})_{0,i}$. 
By looking at dimensions, the only bidegree not fully accounted for by this process is $(1,1)$. 
The extra basis element in bidegree $(1,1)$ spans a 1-dimensional module with bigraded character $qt\chi_2$. 
Thus Alfano concluded the following.
\begin{theorem}[\!\!{\cite[equation (4.19)]{Alfano}}]
The bigraded character of $R^{(2,0)}_{\mathfrak{I}_{2}(n)}$ is given by
\begin{equation}\label{eq:alfano-(2,0)}
\begin{aligned}
    \Char(R_{\mathfrak{I}_{2}(n)}^{(2,0)}; q,t) =&\ \chi_1 + qt\chi_2 + \sum_{i=0}^n q^{n-i}t^i \chi_2 + \sum_{i=1}^{\lfloor \frac{n-1}{2} \rfloor} \sum_{h=0}^i q^{i-h}t^h \chi^i + \sum_{i=1}^{\lfloor \frac{n-1}{2} \rfloor} \sum_{h=0}^{n-i} q^{n-i-h}t^h \chi^i \\
    &+ \begin{cases}
        \sum_{i=0}^{\frac{n}{2}}q^{\frac{n}{2}-i}t^i(\chi_3 + \chi_4)& \text{ if $n$ even,}\\
        0& \text{ if $n$ odd.}
    \end{cases}
\end{aligned}
\end{equation}
\end{theorem}

We are ready to prove our first main result.

\begin{reptheorem}{thm:b-f-character}
Let $\mathbf{q} = (q_1,\ldots, q_k)$ and $\mathbf{u} = (u_1,\ldots, u_j)$. The $\mathbf{q},\mathbf{u}$-graded character series for $R^{(k,j)}_{\mathfrak{I}_{2}(n)}$ is given by
\begin{equation}
    \begin{aligned}
        \Char(R_{\mathfrak{I}_{2}(n)}^{(k,j)}; \mathbf{q};\mathbf{u}) = &\ \chi_1 + s_{(1,1)}(\mathbf{q}/\mathbf{u})\chi_2 + s_{(n)}(\mathbf{q}/\mathbf{u})\chi_2+ \sum_{i=1}^{\lfloor \frac{n-1}{2} \rfloor} (s_{(i)}(\mathbf{q}/\mathbf{u}) + s_{(n-i)}(\mathbf{q}/\mathbf{u})) \chi^i\\
        &+ \begin{cases}
            s_{(\frac{n}{2})}(\mathbf{q}/\mathbf{u})(\chi_3+\chi_4)& \text{ if $n$ even,}\\
        0& \text{ if $n$ odd.}
        \end{cases}
    \end{aligned}
\end{equation}
\end{reptheorem}

\begin{remark}
Using the notation from Remark~\ref{rmk:extra-characters}, we can write the character series more succinctly as
\begin{align}
        \Char(R_{\mathfrak{I}_{2}(n)}^{(k,j)}; \mathbf{q};\mathbf{u}) = &\ \chi_1 + s_{(1,1)}(\mathbf{q}/\mathbf{u})\chi_2 + s_{(n)}(\mathbf{q}/\mathbf{u})\chi_2+ \sum_{i=1}^{n-1} s_{(i)}(\mathbf{q}/\mathbf{u}) \chi^i.
    \end{align}
\end{remark}

\begin{proof}[Proof of Theorem~\ref{thm:b-f-character}]
First observe that equation~(\ref{eq:alfano-(2,0)}) can be rewritten in terms of Schur functions (which are a special case of super Schur functions) as
\begin{equation}\label{eq:new-super}
    \begin{aligned}
        \Char(R_{\mathfrak{I}_{2}(n)}^{(2,0)}; q,t) = &\ \chi_1 + s_{(1,1)}(q,t)\chi_2 + s_{(n)}(q,t)\chi_2+ \sum_{i=1}^{\lfloor \frac{n-1}{2} \rfloor} (s_{(i)}(q,t) + s_{(n-i)}(q,t)) \chi^i\\
        &+ \begin{cases}
            s_{(\frac{n}{2})}(q,t)(\chi_3+\chi_4)& \text{ if $n$ even,}\\
        0& \text{ if $n$ odd.}
        \end{cases}
    \end{aligned}
\end{equation}
This determines the coefficients $c_{\lambda,\mu}$ in equation~(\ref{eq:diagonal-supersymmetry_I}) for all $\lambda \in P(2,0,2)$, i.e., for all $\lambda$ which satisfy $\ell(\lambda) \leq 2$, and for all $\chi^\mu$.
Recall that we require $n \geq 2$.
Explicitly, for $\ell(\lambda) \leq 2$ we have:
\begin{itemize}
    \item if $\lambda = \varnothing$ and $\chi^\mu = \chi_1$, then $c_{\lambda, \mu}=1$;
    \item if $\lambda = (1,1)$ and $\chi^\mu = \chi_2$, then $c_{\lambda, \mu}=1$;
    \item if $\lambda = (n)$ and $\chi^\mu = \chi_2$, then $c_{\lambda, \mu}=1$;
    \item if $n$ is even, if $\lambda = (\frac{n}{2})$ and $\chi^\mu = \chi_3$, then $c_{\lambda, \mu}=1$;
    \item if $n$ is even, if $\lambda = (\frac{n}{2})$ and $\chi^\mu = \chi_4$, then $c_{\lambda, \mu}=1$;
    \item if $\lambda = (i)$ for $i \in \{1,\ldots,\lfloor \frac{n-1}{2} \rfloor\}$ and $\chi^\mu = \chi^i$, then $c_{\lambda, \mu}=1$;
    \item if $\lambda = (n-i)$ for $i \in \{1,\ldots,\lfloor \frac{n-1}{2} \rfloor\}$ and $\chi^\mu = \chi^i$, then $c_{\lambda, \mu}=1$;
    \item otherwise $c_{\lambda, \mu}=0$.
\end{itemize}
The question which remains is if we have found all of the coefficients $c_{\lambda,\mu}$ which appear in the universal series in equation~(\ref{eq:diagonal-supersymmetry_I}).
The condition $\lambda \in P(k,j,2)$ means $\lambda$ such that $\ell(\lambda) \leq 2$ and $\lambda_{k+1} \leq j$, so since we already determined all $c_{\lambda,\mu}$ for all $\lambda$ which satisfy $\ell(\lambda) \leq 2$, there are no additional coefficients left to determine.
Hence expanding equation~(\ref{eq:diagonal-supersymmetry_I}) with the coefficients we have determined completes the proof.
\end{proof}

\section{The multigraded Hilbert series}\label{sec:hilbert-series}

In this section, we derive the Hilbert series of $R^{(k,j)}_{\mathfrak{I}_{2}(n)}$. 
The following theorem was first shown by Alfano \cite[equation (4.17)]{Alfano} for $k=2$ and $j=0$. It was then generalized by Bergeron {\cite[Theorem 2.3 and equation (30)]{Bergeron2013}} to arbitrary $k$ with $j=0$.

\begin{theorem}\label{thm:b-f-hilbert}
 Let $\mathbf{q} = (q_1,\ldots, q_k)$ and $\mathbf{u} = (u_1,\ldots, u_j)$. The $\mathbf{q},\mathbf{u}$-graded Hilbert series for $R^{(k,j)}_{\mathfrak{I}_{2}(n)}$ is given by
    \begin{equation}
    \Hilb(R_{\mathfrak{I}_{2}(n)}^{(k,j)}; \mathbf{q};\mathbf{u})
 = 1 + s_{(1,1)}(\mathbf{q}/\mathbf{u}) + s_{(n)}(\mathbf{q}/\mathbf{u}) + 2\sum_{i=1}^{n-1} s_{(i)}(\mathbf{q}/\mathbf{u}).
    \end{equation}
\end{theorem}

\begin{proof}
This follows immediately from Theorem~\ref{thm:b-f-character} and the fact that the characters $\chi_1,\chi_2,\chi_3,\chi_4$ are 1-dimensional, and the characters $\chi^i$ are 2-dimensional.
\end{proof}

By specializing the Hilbert series at $q_1, \ldots, q_k, u_1, \ldots, u_j = 1$, we obtain a formula for the dimension of $R^{(k,j)}_{\mathfrak{I}_{2}(n)}$.

\begin{corollary}\label{cor:b-f-cardinality-second}
    The dimension of $R^{(k,j)}_{\mathfrak{I}_{2}(n)}$ is given by
    \begin{equation}\label{eq:hilb-cardinality}
    1+ \binom{k}{2} + kj + \binom{j+1}{2} + \sum_{h=0}^n \binom{j}{h}\binom{k+n-h-1}{n-h} + 2\sum_{i=1}^{n-1} \sum_{h=0}^i \binom{j}{h}\binom{k+i-h-1}{i-h}.
    \end{equation}
\end{corollary}

\begin{proof}
Consider the Hilbert series given in Theorem~\ref{thm:b-f-hilbert}.
Recall the Schur function specializations
\begin{equation} s_{(1^m)}(\mathbf{q})|_{q_1, \ldots, q_k =1} = \binom{k}{m} \text{ and } s_{(m)}(\mathbf{q})|_{q_1, \ldots, q_k =1} = \binom{k+m-1}{m}.\end{equation}
Then we obtain the super Schur function specializations
\begin{equation}\label{eq:super-schur-specialization} s_{(m)}(\mathbf{q}/\mathbf{u})|_{q_1, \ldots, q_k,  u_1, \ldots, u_j = 1} = \sum_{\ell = 0}^m \binom{k+\ell-1}{\ell}\binom{j}{m-\ell},\end{equation}
and 
\begin{equation} s_{(1^m)}(\mathbf{q}/\mathbf{u})|_{q_1, \ldots, q_k, u_1, \ldots, u_j = 1} = \sum_{\ell = 0}^m \binom{k}{\ell}\binom{j+m-\ell-1}{m-\ell}.\end{equation}
Specializing the Hilbert series at $q_1, \ldots, q_k, u_1, \ldots, u_j =1$ proves the claim.
\end{proof}

\section{A monomial basis}\label{sec:monomial-basis}

In this section, we derive a monomial basis for $R_{\mathfrak{I}_2(n)}^{(k,j)}$.
In lieu of determining the $\mathfrak{I}_{2}(n)$-invariants of the polynomial ring $\C[\bm{x}^{(1)}, \ldots, \bm{x}^{(k)}, \bm{\theta}^{(1)}, \ldots, \bm{\theta}^{(j)}]$, we list some elements of $I_{\mathfrak{I}_{2}(n)}^{(k,j)}$ which we will use as straightening relations in the quotient ring $R_{\mathfrak{I}_{2}(n)}^{(k,j)}$.

\begin{proposition}\label{prop:ideal-elements}
    The following elements are in the ideal $I_{\mathfrak{I}_{2}(n)}^{(k,j)}$:
    \begin{equation}
    \begin{aligned}
        &\{x_1^{(h)}x_1^{(i)} + x_2^{(h)}x_2^{(i)}\, |\, 1 \leq h \leq i \leq k\} \cup
        \{x_1^{(h)}\theta_1^{(i)} + x_2^{(h)}\theta_2^{(i)}\, |\, 1 \leq h \leq k, 1 \leq i \leq j\}\\ 
        &\cup \{\theta_1^{(h)}\theta_1^{(i)} + \theta_2^{(h)}\theta_2^{(i)}\, |\, 1 \leq h < i \leq j\}\\
        &\cup \{ (x_1^{(1)})^{\alpha_1} \cdots (x_1^{(k)})^{\alpha_k}(\theta_1^{(1)})^{\beta_1} \cdots (\theta_1^{(j)})^{\beta_j}\, |\, \alpha_1 + \cdots + \alpha_k + \beta_1 + \cdots + \beta_j = n\},
    \end{aligned}
    \end{equation}
    where the $\alpha_i$ are nonnegative integers and the $\beta_i$ are $0$ or $1$.
\end{proposition}

\begin{proof}
The set $\{(x_1^{(1)})^2 + (x_2^{(1)})^2, \Re(x_1^{(1)} + i x_2^{(1)})^n\}$ is a minimal set of generators of $I_{\mathfrak{I}_{2}(n)}^{(1,0)}$. 
Using equation~(\ref{eq:real-expansion}), since $(x_{1}^{(1)})^2 \equiv -(x_{2}^{(1)})^2 \pmod {{I_{\mathfrak{I}_{2}(n)}^{(1,0)}}}$, the invariant $\Re(x_1^{(1)} + i x_2^{(1)})^n$ can be replaced by $(x_{1}^{(1)})^n$.
Thus $\{(x_1^{(1)})^2 + (x_2^{(1)})^2, (x_1^{(1)})^n\}$ is also a minimal set of generators for $I_{\mathfrak{I}_{2}(n)}^{(1,0)}$.\footnote{Here and throughout, whenever we are considering $I_{\mathfrak{I}_2(n)}^{(0,j)}$ or $R_{\mathfrak{I}_2(n)}^{(0,j)}$, if we use $I_{\mathfrak{I}_2(n)}^{(1,0)}$ or $R_{\mathfrak{I}_2(n)}^{(1,0)}$ in a proof, we can first construct $I_{\mathfrak{I}_2(n)}^{(1,j)}$ or $R_{\mathfrak{I}_2(n)}^{(1,j)}$ and then restrict to the desired ideal or ring.}

Applying all sequences of operators $E_\ell$ from equations~(\ref{eq:bosonic-E}-\ref{eq:fermionic-E}) to $(x_1^{(1)})^2 +(x_2^{(1)})^2$ gives the sets 
\begin{align}\label{eq:three_sets}
&\{x_1^{(h)}x_1^{(i)} + x_2^{(h)}x_2^{(i)}\, |\, 1 \leq h \leq i \leq k\},\\
&\{x_1^{(h)}\theta_1^{(i)} + x_2^{(h)}\theta_2^{(i)}\, |\, 1 \leq h \leq k, 1 \leq i \leq j\},\\ 
&\{\theta_1^{(h)}\theta_1^{(i)} + \theta_2^{(h)}\theta_2^{(i)}\, |\, 1 \leq h < i \leq j\}.\label{eq:order-matters}\end{align}
Regarding equation~\ref{eq:order-matters}, the operators must be applied in a certain order to avoid squaring any fermionic variables. Specifically, apply $E_{k+i-1}\cdots E_1$ to $(x_1^{(1)})^2+(x_2^{(1)})^2$ to yield $x_1^{(1)}\theta_1^{(i)}+x_2^{(1)}\theta_2^{(i)}$. Then apply $E_{k+h-1} \cdots E_1$ to yield $\theta_1^{(h)}\theta_1^{(i)} + \theta_2^{(h)}\theta_2^{(i)}$.

Similarly, applying all sequences of operators $E_\ell$ to $(x_{1}^{(1)})^n$ gives the set
\begin{equation}\label{eq:fourth_set}\{ (x_1^{(1)})^{\alpha_1} \cdots (x_1^{(k)})^{\alpha_k}(\theta_1^{(1)})^{\beta_1} \cdots (\theta_1^{(j)})^{\beta_j}\, |\, \alpha_1 + \cdots + \alpha_k + \beta_1 + \cdots + \beta_j = n,  \alpha_i \in \Z_{\geq 0}, \beta_i \in \{0,1\}\}.\end{equation}

Since $(x_{1}^{(1)})^2 +(x_{2}^{(1)})^2$ and $(x_{1}^{(1)})^n$ are in $I_{\mathfrak{I}_{2}(n)}^{(k,j)}$ and the operators $E_\ell$ are $\mathfrak{I}_{2}(n)$-invariant, the polynomials in equations~(\ref{eq:three_sets}-\ref{eq:fourth_set}) are in the ideal ${I_{\mathfrak{I}_{2}(n)}^{(k,j)}}$.
\end{proof}

Now we find a monomial basis for $R^{(k,j)}_{\mathfrak{I}_{2}(n)}$.

\begin{theorem}\label{thm:b-f-basis}
    A monomial basis $B^{(k,j)}_{\mathfrak{I}_{2}(n)}$ for the coinvariant ring $R^{(k,j)}_{\mathfrak{I}_{2}(n)}$ is given by:
    \begin{equation}\label{eq:basis-set}
    \begin{aligned}
        &\{(x_{1}^{(1)})^{\alpha_1} \cdots (x_{1}^{(k)})^{\alpha_k}(\theta_{1}^{(1)})^{\beta_1} \cdots (\theta_{1}^{(j)})^{\beta_j} \, | \, 0 \leq \alpha_1 + \cdots + \alpha_k + \beta_1 + \cdots + \beta_j \leq n-1\}\\
        &\cup \bigcup_{i=1}^k \{(x_{1}^{(i)})^{\alpha_i} \cdots (x_{1}^{(k)})^{\alpha_k} (\theta_{1}^{(1)})^{\beta_1} \cdots (\theta_{1}^{(j)})^{\beta_j} x_{2}^{(i)} \, | \, 0\leq \alpha_i + \cdots + \alpha_k + \beta_1 + \cdots + \beta_j \leq n-1\}\\
        &\cup \bigcup_{i=1}^j \{(\theta_{1}^{(i+1)})^{\beta_{i+1}} \cdots (\theta_{1}^{(j)})^{\beta_j} \theta_{2}^{(i)} \, | \, 0\leq \beta_{i+1} + \cdots + \beta_j \leq n-1\}\\
        &\cup \{x_{1}^{(h)}x_{2}^{(i)}\, | \, 1 \leq h < i \leq k\} \cup \{x_{1}^{(h)}\theta_{2}^{(i)}\, | \, 1 \leq h \leq k, 1 \leq i \leq j\} \cup \{\theta_{1}^{(h)}\theta_{2}^{(i)}\, | \, 1 \leq h \leq i \leq j\},
    \end{aligned} 
    \end{equation}
    where the $\alpha_i$ are nonnegative integers and the $\beta_i$ are $0$ or $1$.
\end{theorem}

\begin{proof}
The main idea of the proof is to describe straightening relations, which will allow us to construct the desired spanning set. 
Then we show that the spanning set is a basis by showing that it has the correct dimension.

The following identities are elements in the ideal ${I_{\mathfrak{I}_{2}(n)}^{(k,j)}}$, by Proposition~\ref{prop:ideal-elements}:
\begin{equation}\label{eq:xxx-relation}
    x_{1}^{(\ell)}(x_{2}^{(h)}x_{1}^{(i)} - x_{1}^{(h)}x_{2}^{(i)}) = x_{2}^{(h)}(x_{1}^{(\ell)}x_{1}^{(i)} + x_{2}^{(\ell)} x_{2}^{(i)}) - x_{2}^{(i)}(x_{1}^{(\ell)}x_{1}^{(h)} + x_{2}^{(\ell)}x_{2}^{(h)}),
\end{equation} 
\begin{equation}\label{eq:thetaxx-relation}
    \theta_{1}^{(\ell)}(x_{2}^{(h)}x_{1}^{(i)} - x_{1}^{(h)}x_{2}^{(i)}) = x_{2}^{(h)}(\theta_{1}^{(\ell)}x_{1}^{(i)} + \theta_{2}^{(\ell)} x_{2}^{(i)}) - x_{2}^{(i)}(\theta_1^{(\ell)}x_{1}^{(h)} + \theta_{2}^{(\ell)}x_{2}^{(h)}),
\end{equation} 
\begin{equation}\label{eq:xxtheta-relation}
    x_{1}^{(\ell)}(x_{2}^{(h)}\theta_{1}^{(i)} - x_{1}^{(h)}\theta_{2}^{(i)}) = x_{2}^{(h)}(x_{1}^{(\ell)}\theta_{1}^{(i)} + x_{2}^{(\ell)} \theta_{2}^{(i)}) - \theta_{2}^{(i)}(x_{1}^{(\ell)}x_{1}^{(h)} + x_{2}^{(\ell)}x_{2}^{(h)}),
\end{equation} 
\begin{equation}\label{eq:thetaxtheta-relation}
    \theta_{1}^{(\ell)}(x_{2}^{(h)}\theta_{1}^{(i)} - x_{1}^{(h)}\theta_{2}^{(i)}) = x_{2}^{(h)}(\theta_{1}^{(\ell)}\theta_{1}^{(i)} + \theta_{2}^{(\ell)} \theta_{2}^{(i)}) - (\theta_1^{(\ell)}x_{1}^{(h)} + \theta_{2}^{(\ell)}x_{2}^{(h)})\theta_{2}^{(i)},
\end{equation} 
\begin{equation}\label{eq:xthetatheta-relation}
    x_{1}^{(\ell)}(\theta_{2}^{(h)}\theta_{1}^{(i)} - \theta_{1}^{(h)}\theta_{2}^{(i)}) = \theta_{2}^{(h)}(x_{1}^{(\ell)}\theta_{1}^{(i)} + x_{2}^{(\ell)} \theta_{2}^{(i)}) - (x_{1}^{(\ell)}\theta_{1}^{(h)} + x_{2}^{(\ell)}\theta_{2}^{(h)})\theta_{2}^{(i)},
\end{equation} 
\begin{equation}\label{eq:thetathetatheta-relation}
    \theta_{1}^{(\ell)}(\theta_{2}^{(h)}\theta_{1}^{(i)} - \theta_{1}^{(h)}\theta_{2}^{(i)}) = -\theta_{2}^{(h)}(\theta_{1}^{(\ell)}\theta_{1}^{(i)} + \theta_{2}^{(\ell)} \theta_{2}^{(i)}) - (\theta_1^{(\ell)}\theta_{1}^{(h)} + \theta_{2}^{(\ell)}\theta_{2}^{(h)})\theta_{2}^{(i)}.
\end{equation} 

We establish the following straightening rules:

\begin{enumerate}
    \item[(A)] $\{x_1^{(h)}x_1^{(i)} + x_2^{(h)}x_2^{(i)}\, |\, 1 \leq h \leq i \leq k\} \cup
        \{x_1^{(h)}\theta_1^{(i)} + x_2^{(h)}\theta_2^{(i)}\, |\, 1 \leq h \leq k, 1 \leq i \leq j\} \cup \{\theta_1^{(h)}\theta_1^{(i)} + \theta_2^{(h)}\theta_2^{(i)}\, |\, 1 \leq h < i \leq j\} \subseteq {I_{\mathfrak{I}_{2}(n)}^{(k,j)}}$ allows us to re-express any monomial $M \in \C[\bm{x}^{(1)}, \ldots, \bm{x}^{(k)},\bm{\theta}^{(1)}, \ldots, \bm{\theta}^{(j)}]$ modulo ${I_{\mathfrak{I}_{2}(n)}^{(k,j)}}$ as a monomial of total degree no greater than 1 in $\{x_{2}^{(1)}, \ldots, x_{2}^{(k)}, \theta_2^{(1)}, \ldots, \theta_2^{(j)}\}$.  
    \item[(B)] $\{ (x_1^{(1)})^{\alpha_1} \cdots (x_1^{(k)})^{\alpha_k}(\theta_1^{(1)})^{\beta_1} \cdots (\theta_1^{(j)})^{\beta_j}\, |\, \alpha_1 + \cdots + \alpha_k + \beta_1 + \cdots + \beta_j = n,  \alpha_i \in \Z_{\geq 0}, \beta_i \in \{0,1\}\} \subseteq {I_{\mathfrak{I}_{2}(n)}^{(k,j)}}$ allows us to re-express any monomial $M \in \C[\bm{x}^{(1)}, \ldots, \bm{x}^{(k)},\bm{\theta}^{(1)}, \ldots, \bm{\theta}^{(j)}]$ modulo ${I_{\mathfrak{I}_{2}(n)}^{(k,j)}}$ as a monomial of total degree no greater than $n-1$ in $\{x_{1}^{(1)}, \ldots, x_{1}^{(k)},\theta_1^{(1)}, \ldots,\theta_1^{(j)}\}$.
    \item[(C)] For any monomial $M \in \C[\bm{x}^{(1)}, \ldots, \bm{x}^{(k)},\bm{\theta}^{(1)}, \ldots, \bm{\theta}^{(j)}]$ which has as a factor $x_{1}^{(\ell)}(x_{1}^{(h)}x_{2}^{(i)})$ for some $\ell$ and $h<i$, then modulo ${I_{\mathfrak{I}_{2}(n)}^{(k,j)}}$ we can re-express that factor within $M$ as $x_{1}^{(\ell)}(x_{2}^{(h)}x_{1}^{(i)})$, by equation~(\ref{eq:xxx-relation}). Similarly, for any monomial $M$ which has as a factor $\theta_{1}^{(\ell)}(x_{1}^{(h)}x_{2}^{(i)})$ for some $\ell$ and $h<i$, then modulo ${I_{\mathfrak{I}_{2}(n)}^{(k,j)}}$ we can re-express that factor within $M$ as $\theta_{1}^{(\ell)}(x_{2}^{(h)}x_{1}^{(i)})$, by equation~(\ref{eq:thetaxx-relation}).
    \item[(D)] For any monomial $M \in \C[\bm{x}^{(1)}, \ldots, \bm{x}^{(k)},\bm{\theta}^{(1)}, \ldots, \bm{\theta}^{(j)}]$ which has as a factor $x_{1}^{(\ell)}(x_{1}^{(h)}\theta_{2}^{(i)})$ for some $\ell$, then modulo ${I_{\mathfrak{I}_{2}(n)}^{(k,j)}}$ we can re-express that factor within $M$ as $x_{1}^{(\ell)}(x_{2}^{(h)}\theta_{1}^{(i)})$, by equation~(\ref{eq:xxtheta-relation}). Similarly, for any monomial $M$ which has as a factor $\theta_{1}^{(\ell)}(x_{1}^{(h)}\theta_{2}^{(i)})$ for some $\ell$, then modulo ${I_{\mathfrak{I}_{2}(n)}^{(k,j)}}$ we can re-express that factor within $M$ as $\theta_{1}^{(\ell)}(x_{2}^{(h)}\theta_{1}^{(i)})$, by equation~(\ref{eq:thetaxtheta-relation}).
    \item[(E)] For any monomial $M \in \C[\bm{x}^{(1)}, \ldots, \bm{x}^{(k)},\bm{\theta}^{(1)}, \ldots, \bm{\theta}^{(j)}]$ which has as a factor $x_{1}^{(\ell)}(\theta_{1}^{(h)}\theta_{2}^{(i)})$ for some $\ell$ and $h<i$, then modulo ${I_{\mathfrak{I}_{2}(n)}^{(k,j)}}$ we can re-express that factor within $M$ as $x_{1}^{(\ell)}(\theta_{2}^{(h)}\theta_{1}^{(i)})$, by equation~(\ref{eq:xthetatheta-relation}). Similarly, for any monomial $M$ which has as a factor $\theta_{1}^{(\ell)}(\theta_{1}^{(h)}\theta_{2}^{(i)})$ for some $\ell$ and $h<i$, then modulo ${I_{\mathfrak{I}_{2}(n)}^{(k,j)}}$ we can re-express that factor within $M$ as $\theta_{1}^{(\ell)}(\theta_{2}^{(h)}\theta_{1}^{(i)})$, by equation~(\ref{eq:thetathetatheta-relation}).
    \item[(F)] For any monomial $M \in \C[\bm{x}^{(1)}, \ldots, \bm{x}^{(k)},\bm{\theta}^{(1)}, \ldots, \bm{\theta}^{(j)}]$ which has as a factor $\theta_{1}^{(\ell)}(\theta_{1}^{(i)}\theta_{2}^{(i)})$ for some $\ell$, then modulo ${I_{\mathfrak{I}_{2}(n)}^{(k,j)}}$ we can re-express that factor within $M$ as $\theta_{1}^{(\ell)}(\theta_{2}^{(i)}\theta_{1}^{(i)})$, by equation~(\ref{eq:thetathetatheta-relation}) specialized at $h=i$. On the other hand, $\theta_{1}^{(\ell)}(\theta_{1}^{(i)}\theta_{2}^{(i)})= -\theta_{1}^{(\ell)}(\theta_{2}^{(i)}\theta_{1}^{(i)})$ by directly swapping the order of the anticommuting variables. Together, these imply that $\theta_{1}^{(\ell)}(\theta_{2}^{(i)}\theta_{1}^{(i)})$ must be $0$ modulo ${I_{\mathfrak{I}_{2}(n)}^{(k,j)}}$, so $M = 0$ modulo ${I_{\mathfrak{I}_{2}(n)}^{(k,j)}}$.
\end{enumerate}

First, note the spanning set should contain a monomial $(x_{1}^{(1)})^{\alpha_1} \cdots (x_{1}^{(k)})^{\alpha_k} (\theta_{1}^{(1)})^{\beta_1} \cdots (\theta_{1}^{(j)})^{\beta_j}$ for each choice of nonnegative integers $\alpha_1, \ldots, \alpha_k$ and $\beta_1,\ldots,\beta_j \in \{0,1\}$ such that $0 \leq \alpha_1 +  \cdots + \alpha_k +\beta_1 + \cdots +\beta_j \leq n-1$ because of (B). While there might be no variables from $\{x_{2}^{(1)}, \ldots, x_{2}^{(k)}, \theta_{2}^{(1)}, \ldots, \theta_{2}^{(j)}\}$ present, there could also be one, say $x_{2}^{(1)}$, of degree 1, due to (A). Thus in the spanning set we include
\begin{equation}\label{eq:23}\{(x_{1}^{(1)})^{\alpha_1} \cdots (x_{1}^{(k)})^{\alpha_k}(\theta_{1}^{(1)})^{\beta_1} \cdots (\theta_{1}^{(j)})^{\beta_j} \, | \, 0 \leq \alpha_1 + \cdots + \alpha_k + \beta_1 + \cdots + \beta_j \leq n-1\},\end{equation}
and 
\begin{equation}\{(x_{1}^{(1)})^{\alpha_1} \cdots (x_{1}^{(k)})^{\alpha_k}(\theta_{1}^{(1)})^{\beta_1} \cdots (\theta_{1}^{(j)})^{\beta_j}x_{2}^{(1)} \, | \, 0 \leq \alpha_1 + \cdots + \alpha_k + \beta_1 + \cdots + \beta_j \leq n-1\}.\end{equation}

In most circumstances, we will be able to swap for a different choice of the variable $x_{2}^{(i)}$ (to lower $i$) due to (C), (D), and (E). The exceptions to (C) occur when $x_{1}^{(h)}x_{2}^{(i)}$ appears alone, so we must also include these monomials in the spanning set:
\begin{equation} \{x_{1}^{(h)}x_{2}^{(i)}\, | \, 1 \leq h < i \leq k\}.\end{equation}
The exceptions to (D) occur when $x_{1}^{(h)}\theta_{2}^{(i)}$ appears alone, so we must also include these monomials in the spanning set:
\begin{equation}\{x_{1}^{(h)}\theta_{2}^{(i)}\, | \, 1 \leq h \leq k, 1 \leq i \leq j\}.\end{equation}
The exceptions to (E) occur when $\theta_{1}^{(h)}\theta_{2}^{(i)}$ appears alone, so we must also include these monomials in the spanning set:
\begin{equation} \{\theta_{1}^{(h)}\theta_{2}^{(i)}\, | \, 1 \leq h < i \leq j\}.\end{equation}

Now we can reduce the problem repeatedly, as there are no more monomials needed left with a variable $x_{1}^{(1)}$ or $x_{2}^{(1)}$. As we already took care of all monomials with no $x_{2}^{(i)}$ variables, we are left to consider
\begin{equation}
\begin{aligned}&\{(x_{1}^{(2)})^{\alpha_2} \cdots (x_{1}^{(k)})^{\alpha_k} (\theta_{1}^{(1)})^{\beta_1} \cdots (\theta_{1}^{(j)})^{\beta_j} x_{2}^{(2)} \, | \, 0\leq \alpha_2 + \cdots + \alpha_k + \beta_1 + \cdots + \beta_j \leq n-1\},\\
&\{(x_{1}^{(3)})^{\alpha_3} \cdots (x_{1}^{(k)})^{\alpha_k} (\theta_{1}^{(1)})^{\beta_1} \cdots (\theta_{1}^{(j)})^{\beta_j} x_{2}^{(3)} \, | \, 0\leq \alpha_3 + \cdots + \alpha_k + \beta_1 + \cdots + \beta_j \leq n-1\},\\
&\qquad\ldots \\
&\{(x_{1}^{(k)})^{\alpha_k} (\theta_{1}^{(1)})^{\beta_1} \cdots (\theta_{1}^{(j)})^{\beta_j} x_{2}^{(k)} \, | \, 0\leq \alpha_k + \beta_1 + \cdots + \beta_j \leq n-1\}.
\end{aligned}
\end{equation}
Now at this point, there are no more monomials needed left with a variable $x_{1}^{(\ell)}$ or $x_{2}^{(\ell)}$ for any $\ell$. We are left to consider monomials consisting of only fermionic variables which must contain some variable $\theta_2^{(\ell)}$. 
By (F), when $\theta_{1}^{(i)}\theta_{2}^{(i)}$ appears alone, we must also include these monomials in the spanning set:
\begin{equation} \{\theta_{1}^{(i)}\theta_{2}^{(i)}\, | \, 1 \leq i \leq j\}.\end{equation}
Going forwards, we cannot have $\theta_1^{(\ell)}\theta_2^{(\ell)}$ appearing in a monomial, as by (F) it will reduce to 0 or be one of the monomials we already included. Thus we are left to progressively consider
\begin{equation}\label{eq:30}
\begin{aligned}&\{(\theta_{1}^{(2)})^{\beta_{2}} \cdots (\theta_{1}^{(j)})^{\beta_j} \theta_{2}^{(1)} \, | \, 0\leq \beta_{2} + \cdots + \beta_j \leq n-1\},\\
&\{(\theta_{1}^{(3)})^{\beta_{3}} \cdots (\theta_{1}^{(j)})^{\beta_j} \theta_{2}^{(2)} \, | \, 0\leq \beta_{3} + \cdots + \beta_j \leq n-1\},\\
&\qquad\ldots\\
&\{(\theta_{1}^{(j)})^{\beta_j} \theta_{2}^{(j-1)} \, | \, 0\leq \beta_j \leq n-1\},\\
&\{ \theta_{2}^{(j)} \}.\end{aligned}
\end{equation}
Taking the union of all of these sets in equations~(\ref{eq:23}-\ref{eq:30}) shows that the claimed basis is a spanning set.

We will show that the spanning set is a basis by showing that its cardinality matches the dimension of $R_{\mathfrak{I}_2(n)}^{(k,j)}$ from Corollary~\ref{cor:b-f-cardinality-second}.
We will show this by induction on $n$.
For fixed $k,j$, let $P_{n}$ denote the cardinality given in equation~(\ref{eq:hilb-cardinality}), and the induction hypothesis is that the cardinality of the spanning set in equation~(\ref{eq:basis-set}) is equal to $P_n$.

For the base case of $n=2$, a direct computation shows that both $P_2$ and the cardinality of the spanning set are
\begin{equation}
    1 + 2k+2j+2kj +\binom{k}{2}+\binom{k+1}{2}+\binom{j}{2}+\binom{j+1}{2}.
\end{equation}

For the inductive step, we calculate
\begin{equation}\label{eq:P-difference}
    P_{n+1} - P_{n} = \sum_{h=0}^{n+1} \binom{j}{h} \binom{k+n-h}{n+1-h} + \sum_{h=0}^{n} \binom{j}{h} \binom{k+n-h-1}{n-h}.
\end{equation}
Let us label the constituent sets in equation~(\ref{eq:basis-set}) which depend on $n$:
\begin{equation}
    A_n =\{(x_{1}^{(1)})^{\alpha_1} \cdots (x_{1}^{(k)})^{\alpha_k}(\theta_{1}^{(1)})^{\beta_1} \cdots (\theta_{1}^{(j)})^{\beta_j} \, | \, 0 \leq \alpha_1 + \cdots + \alpha_k + \beta_1 + \cdots + \beta_j \leq n-1\}\
\end{equation}
\begin{equation}
    B_n = \bigcup_{i=1}^k \{(x_{1}^{(i)})^{\alpha_i} \cdots (x_{1}^{(k)})^{\alpha_k} (\theta_{1}^{(1)})^{\beta_1} \cdots (\theta_{1}^{(j)})^{\beta_j} x_{2}^{(i)} \, | \, 0\leq \alpha_i + \cdots + \alpha_k + \beta_1 + \cdots + \beta_j \leq n-1\}
\end{equation}
\begin{equation}
    C_n = \bigcup_{i=1}^j \{(\theta_{1}^{(i+1)})^{\beta_{i+1}} \cdots (\theta_{1}^{(j)})^{\beta_j} \theta_{2}^{(i)} \, | \, 0\leq \beta_{i+1} + \cdots + \beta_j \leq n-1\}
\end{equation}

We compute the set differences between $n+1$ and $n$, finding that
\begin{equation}
    A_{n+1}\setminus A_n = \{(x_{1}^{(1)})^{\alpha_1} \cdots (x_{1}^{(k)})^{\alpha_k}(\theta_{1}^{(1)})^{\beta_1} \cdots (\theta_{1}^{(j)})^{\beta_j} \, | \, \alpha_1 + \cdots + \alpha_k + \beta_1 + \cdots + \beta_j = n\},
\end{equation}
\begin{equation}
    B_{n+1}\setminus B_n = \bigcup_{i=1}^k \{(x_{1}^{(i)})^{\alpha_i} \cdots (x_{1}^{(k)})^{\alpha_k} (\theta_{1}^{(1)})^{\beta_1} \cdots (\theta_{1}^{(j)})^{\beta_j} x_{2}^{(i)} \, | \, \alpha_i + \cdots + \alpha_k + \beta_1 + \cdots + \beta_j = n\},
\end{equation}
\begin{equation}
    C_{n+1}\setminus C_n = \bigcup_{i=1}^j \{(\theta_{1}^{(i+1)})^{\beta_{i+1}} \cdots (\theta_{1}^{(j)})^{\beta_j} \theta_{2}^{(i)} \, | \, \beta_{i+1} + \cdots + \beta_j = n\}.
\end{equation}
Counting with stars and bars arguments, their cardinalities are 
\begin{equation}\label{eq:A-final}
    |A_{n+1}\setminus A_n| = \sum_{h=0}^n \binom{j}{h} \binom{k-1+n-h}{n-h},
\end{equation}
\begin{equation}
    |B_{n+1}\setminus B_n| = \sum_{i=1}^k\sum_{h=0}^{n} \binom{j}{h} \binom{k-i+n-h}{n-h},
\end{equation}
which by reversing the order of summation and the hockey-stick identity simplifies to
\begin{equation}\label{eq:B-final}
    |B_{n+1}\setminus B_n| = \sum_{h=0}^{n} \binom{j}{h} \binom{k+n-h}{n-h-1}.
\end{equation}
We compute that
\begin{equation}\label{eq:C-final}
    |C_{n+1}\setminus C_n| = \sum_{i=1}^j\binom{j-i}{n} = \binom{j}{n+1}.
\end{equation}
Note that equations~(\ref{eq:B-final}) and~(\ref{eq:C-final}) can be combined to give
\begin{equation}
    \sum_{h=0}^{n+1} \binom{j}{h} \binom{k+n-h}{n-h-1},
\end{equation}
hence their sum with equation~(\ref{eq:A-final}) equals equation~(\ref{eq:P-difference}).
Thus the spanning set has cardinality $P_n$ for all $n\geq 2$, and hence must be a basis.
\end{proof}

\begin{remark}
    It is possible to prove the main results without using diagonal supersymmetry, generalizing the method in \cite[Chapter IV]{Alfano}.
    First, derive the monomial basis for $R^{(k,j)}_{\mathfrak{I}_{2}(n)}$ (Theorem~\ref{thm:b-f-basis}) by constructing a spanning set, and then construct enough linearly independent harmonics in the harmonic space isomorphic to $R_{\mathfrak{I}_2(n)}^{(k,j)}$. 
    Second, enumerate each graded component of the basis to determine the explicit multigraded Hilbert series for $R^{(k,j)}_{\mathfrak{I}_{2}(n)}$ (Theorem~\ref{thm:b-f-hilbert}).
    Finally, apply the polarization operators $E_\ell$ to Alfano's formula for the bigraded character series of $R_{\mathfrak{I}_2(n)}^{(2,0)}$ in equation~(\ref{eq:alfano-(2,0)}) to derive the explicit multigraded character series for $R^{(k,j)}_{\mathfrak{I}_{2}(n)}$ (Theorem~\ref{thm:b-f-character}).
\end{remark}

\section{The sign character}\label{sec:sign-character}

The sign character, which is $\chi_2$ in our notation, is of particular interest, so we briefly record its multigraded multiplicty in $R_{\mathfrak{I}_2(n)}^{(k,j)}$. The dimension of the sign character of $R_W^{(2,0)}$ defines the Catalan numbers of type $W$, for any finite Coxeter group $W$, and the $q,t$-graded dimension of the sign character defines the $q,t$-Catalan numbers of type $W$. Denote the latter by $\Cat(W;q,t)$. See \cite[Section 7]{Haiman1994} and \cite{Stump} for a discussion of what is conjectured and known for different groups $W$. 

Define
\begin{equation} [n]_{q,t} := q^{n-1} + q^{n-2}t+ \cdots + qt^{n-2} + t^{n-1}.\end{equation}
From Alfano's formula for the $q,t$-graded character series of $R_{\mathfrak{I}_2(n)}^{(2,0)}$ \cite{Alfano}, we can extract the $q,t$-graded multiplicity of the sign character as
\begin{equation}
    \Cat(\mathfrak{I}_2(n);q,t) = [n+1]_{q,t} + qt = s_{(n)}(q,t) + s_{(1,1)}(q,t),
\end{equation}
which is also given in \cite[Corollary 3]{Stump}.

More generally, for arbitrary $(k,j)$, denote the $\mathbf{q},\mathbf{u}$-graded multiplicity of the sign character of $R_W^{(k,j)}$ by $\Cat(W;\mathbf{q};\mathbf{u})$,\footnote{Note that $\Cat(W;q,t)$ is recovered, when $k \geq 2$, by setting all but two of the $q_i$ to zero and all of the $u_i$ to zero. Specializing this with all variables to 1 is not in general a Catalan number of type $W$.} where as usual, $\mathbf{q} = q_1,\ldots, q_k$ and $\mathbf{u} = u_1,\ldots, u_j$. Then from Theorem~\ref{thm:b-f-character}, we immediately get that
\begin{equation}
    \Cat(\mathfrak{I}_2(n);\mathbf{q};\mathbf{u}) = s_{(n)}(\mathbf{q}/\mathbf{u}) + s_{(1,1)}(\mathbf{q}/\mathbf{u}).
\end{equation}

\section{Numerology}\label{sec:discussion}

In this section, we discuss various specializing the dimension of $R^{(k,j)}_{G}$ to some small values of $(k,j)$, using Corollary~\ref{cor:b-f-cardinality-second} to get formulas for the dimension of $R^{(k,j)}_{\mathfrak{I}_2(n)}$. See Figure~\ref{fig:typeI-dims} for some values of $\dim(R_{\mathfrak{I}_2(n)}^{(k,j)})$.

\begin{figure}

\begin{tabular}{l|lll}
\diagbox{$k$}{$j$} & $0$       & $1$    & $2$   \\ \hline
$0$   &           & $4$    & $\leq10$   \\
$1$   & $2n$      & $4n+1$ & $8n$   \\
$2$   & $(n+1)^2$ & $2(n^2+n+2)$  &  $4n^2+9$   \\
$3$   & $\frac{2n^3+9n^2+13n+18}{6}$  &  $\frac{2n^3+6n^2+7n+21}{3}$  &   $\frac{4n^3+6n^2+8n+36}{3}$    \\
$4$   & $\frac{n^4+8n^3+23n^2+28n+72}{12}$  & $\frac{n^4+6n^3+14n^2+15n+66}{6}$ &  $\frac{n^4 + 4n^3 +8n^2 +8n+51}{3}$ 
\end{tabular}

\vspace{1em}

\begin{tabular}{l|ll}
\diagbox{$k$}{$j$}  & $3$     & $4$ \\ \hline
$0$   &   $\leq21$      &  $\leq41$   \\
$1$   & $\geq16n-8$ &  $\geq32n-35 $   \\
$2$   & $8n^2-8n+20$  &  $\leq 16n^2-32n+50$   \\
$3$   &   $\frac{8n^3+16n+51}{3}$      &  $\frac{16n^3-24n^2+56n+48}{3}$   \\
$4$   & $\frac{2n^4+4n^3+10n^2+8n+72}{3}$   &   $\frac{4n^4+20n^2+99}{3}$ 
\end{tabular}

    \caption{Some polynomial dimension formulas for $R_{\mathfrak{I}_2(n)}^{(k,j)}$ for $0 \leq k,j \leq 4$. When $j-k \geq 2$, the dimensions eventually stabilize to the listed values as $n \to \infty$. All formulas are obtained by specializing Corollary~\ref{cor:b-f-cardinality-second}.}
    \label{fig:typeI-dims}
\end{figure}

There is interest in establishing formulas that are type-invariant (either for all types, or for at least some types) depending on information on $G$. For example, it is a classical result that $\dim (R_G^{(1,0)}) = |G|$ if and only if $G$ is a finite reflection group. See Figure~\ref{fig:type-independent-dims} for some example dimension formulas, which are known or conjectured to hold for both types $A$ and $I$.

\begin{figure}

\begin{tabular}{l|lllll}
\diagbox{$k$}{$j$} & $0$       & $1$    & $2$  \\ \hline
$0$   &           & $2^r$    & $\binom{2r+1}{r}$ \\
$1$   & $|W|$      & $|\Sigma(W,S)|$ & $2^r|W|$\\
$2$   & $(h+1)^r$ &        &  
\end{tabular}

    \caption{Known or conjectured dimension formulas for $R_{W}^{(k,j)}$ for types $A$ and $I$ together. The entries for $(1,0)$, $(0,1)$, and $(0,2)$ are type invariant for all finite Coxeter groups. The entries for $(2,0)$, $(1,1)$, and $(1,2)$ are known or expected to not necessarily hold for all types, but are known or expected to hold for both types $A$ and $I$.}
    \label{fig:type-independent-dims}
\end{figure}

Haiman \cite{Haiman1994} observed that $\dim(R_W^{(2,0)}) \geq (h+1)^r$, where $h$ denotes the Coxeter number and $r$ denotes the rank for many finite Coxeter groups $W$. This inequality was proven by Gordon \cite{Gordon} for all Weyl groups $W$. Alfano showed that equality holds for type $I$, where $\dim(R_{\mathfrak{I}_2(n)}^{(2,0)}) = (n+1)^2$. When $W = \mathfrak{B}_n$, equality holds for $n = 2, 3$, while for $n \geq 4$, Haiman conjectured and Ajila--Griffeth \cite{AjilaGriffeth} proved that $\dim (R_{\mathfrak{B}_n}^{(2,0)}) > (2n+1)^n$. 
For $W$ an irreducible complex reflection group with order greater than 2 reflections, the lower bound was further improved by Griffeth \cite{Griffeth}.

Zabrocki \cite{Zabrocki2020} observed that $\dim(R_W^{(1,1)})$ is the order of $\Sigma(W,S)$, the Coxeter complex of $W$, for small $n$ for $W = \mathfrak{S}_n$, $W = \mathfrak{B}_n$, $W = \mathfrak{D}_n$, and $W = \mathfrak{I}_2(n)$. 
This was proven in type $A$ by Rhoades and Wilson \cite{RhoadesWilson2023} and in type $B$ by Bhattacharya \cite{Bhattacharya}.
Rhoades and Swanson calculated that this is not true for $W = \mathfrak{F}_4$. Specializing Corollary~\ref{cor:b-f-cardinality-second}, we find that $\dim(R_{\mathfrak{I}_{2}(n)}^{(1,1)}) = 4n+1$, which is the order of the Coxeter complex of $\mathfrak{I}_2(n)$. 

It is not hard to see that $\dim(R_W^{(0,1)}) = 2^r$; for example, this was noted by Machacek and Zabrocki \cite{ZabrockiMachacek} for type $A$. A type-independent proof follows as a corollary from the work of Kim and Rhoades \cite{KimRhoades2022} on $R_W^{(0,2)}$, by setting one set of variables to zero. Specializing Corollary~\ref{cor:b-f-cardinality-second}, we find that $\dim(R_{\mathfrak{I}_{2}(n)}^{(0,1)}) = 4$.

Kim and Rhoades \cite{KimRhoades2022} showed that $\dim(R_W^{(0,2)}) = \binom{2r+1}{r}$, for any irreducible complex reflection group $W$. Since $\mathfrak{I}_2(2)$ is reducible, when $n>2$, this agrees with specializing Corollary~\ref{cor:b-f-cardinality-second}, which says that $\dim(R_{\mathfrak{I}_2(n)}^{(k,j)}) = 10 = \binom{5}{2}$.

In \cite{Lentfer2024}, it was observed that for small $n$, $\dim(R_W^{(1,2)}) = 2^r|W|$ for $W = \mathfrak{S}_n$ or $\mathfrak{B}_n$ and conjectural bases were constructed which give those enumerations. 
Specializing Corollary~\ref{cor:b-f-cardinality-second}, we find that $\dim(R_{\mathfrak{I}_{2}(n)}^{(1,2)}) = 8n = 2^2|\mathfrak{I}_{2}(n)|$. 

Haiman observed that for small $n$, $\dim(R_{n}^{(3,0)}) = 2^n(n+1)^{n-2}$. Note that for $W = \mathfrak{I}_{2}(6) = \mathfrak{G}_2$, the factorization $\dim(R_{\mathfrak{I}_{2}(6)}^{(3,0)}) = 142 = 2 \cdot 71$ does not seem to readily mesh with Haiman's conjecture on $R_{n}^{(3,0)}$ in a type-independent way.

\appendix

\section{Cyclic groups}\label{sec:cyclic}
In this appendix we concisely derive analogous results for the cyclic group. Let $\Z_n$ denote the cyclic group of order $n$. As this is a rank 1 complex reflection group, each set of variables only has one variable, so we dispense with the superscripts, and consider the ring $R_{\Z_n}^{(k,j)} = \C[x_1,\ldots,x_k, \theta_1,\ldots, \theta_j]/\langle\C[x_1,\ldots,x_k, \theta_1,\ldots, \theta_j]_+^{\Z_n}\rangle$. The action of $\Z_n$ on $x_i$ is given by $g \cdot x_i = \zeta_n^g x_i$ and on $\theta_i$ by $g \cdot \theta_i = \zeta_n^g \theta_i$ for $g \in \Z_n = \{0,1,\ldots,n-1\}$, where $\zeta_n = e^{2\pi i / n}$ is a principle $n$th root of unity.
The generators and relations for $\Z_n$ may be written as $\langle a\, |\, a^n = 1\rangle$. 
The irreducible characters of $\Z_n$ are:
\begin{itemize}
    \item $\chi_i$ defined by $\chi_i(a^\ell) = \zeta_n^{i\ell}$ for $i \in \{0,\ldots,n-1\}$.
\end{itemize}

We determine the multigraded character series.

\begin{theorem}\label{thm:cyclic-character}
    Let $\mathbf{q} = (q_1,\ldots, q_k)$ and $\mathbf{u} = (u_1,\ldots, u_j)$. The $\mathbf{q},\mathbf{u}$-graded character series for $R^{(k,j)}_{\Z_n}$ is given by
    \begin{equation}
        \Char(R_{\Z_n}^{(k,j)}; \mathbf{q};\mathbf{u}) = \sum_{i=0}^{n-1} s_{(i)}(\mathbf{q}/\mathbf{u}) \chi_{i}.
    \end{equation}
\end{theorem}

\begin{proof}
It is straightforward to show that $R^{(1,0)}_{\Z_n} = \C[x_1]/(x_1^n)$, with monomial basis $\{1,x_1,x_1^2,\ldots, x_1^{n-1}\}$.
Then the singly graded character of $R^{(1,0)}_{\Z_n}$ is the following graded version of the regular representation:
\begin{equation}\label{eq:cyclic-regular-rep}
    \Char(R_{\Z_n}^{(1,0)}; q) =  \sum_{i=0}^{n-1} q^i\chi_{i}.
\end{equation}

Since each set of variables has only one variable in it, then diagonal supersymmetry (Theorem~\ref{thm:G-main-theorem}) says that for partitions $\lambda$ with $\ell(\lambda) \leq 1$, and indices $\mu$ of irreducible $\Z_n$-characters, there exist nonnegative integer coefficients $c_{\lambda,\mu}$ such that for any $(k,j)$, the multigraded character series of $R_{\Z_n}^{(k,j)}$ is
\begin{equation}\label{eq:cyclic-diagonal-supersymmetry}
    \Char(R_{\Z_n}^{(k,j)}; \mathbf{q};\mathbf{u}) =  \sum_{\lambda \in P(k,j,1)} \sum_{\chi^\mu \in \IrrChar(\Z_n)}c_{\lambda,\mu} s_\lambda(\mathbf{q}/\mathbf{u})\chi^\mu.
\end{equation}
Thus equation~(\ref{eq:cyclic-regular-rep}) determines all coefficients $c_{\lambda,\mu}$ where $\lambda \in P(1,0,1)$, i.e., $\ell(\lambda) \leq 1$. 
Explicitly, for $n\geq 1$, and $\ell(\lambda) \leq 1$ we have:
\begin{itemize}
    \item if $\lambda = (i)$ for $i \in \{0,\ldots,n-1\}$ and $\chi^\mu = \chi_i$, then $c_{\lambda, \mu}=1$;
    \item otherwise $c_{\lambda, \mu}=0$.
\end{itemize}
Thus we have determined $c_{\lambda,\mu}$ for any $\lambda \in P(k,j,1)$.
Hence expanding equation~(\ref{eq:cyclic-diagonal-supersymmetry}) with the coefficients we have determined completes the proof.
\end{proof}

We then conclude the following formula for the Hilbert series. This recovers a result of Bergeron for the purely bosonic case (at $j=0$) \cite[Theorem 2.3]{Bergeron2013}.

\begin{theorem}\label{thm:cyclic-hilbert}
    Let $\mathbf{q} = (q_1,\ldots, q_k)$ and $\mathbf{u} = (u_1,\ldots, u_j)$. The $\mathbf{q},\mathbf{u}$-graded Hilbert series for $R^{(k,j)}_{\Z_n}$ is given by
    \begin{equation}
    \Hilb(R_{\Z_n}^{(k,j)}; \mathbf{q};\mathbf{u})
 = \sum_{i=0}^{n-1} s_{(i)}(\mathbf{q}/\mathbf{u}).
    \end{equation}
\end{theorem}

\begin{proof}
This follows immediately from Theorem~\ref{thm:cyclic-character} since each character is 1-dimensional.
\end{proof}

As a corollary, we obtain the dimension of $R_{\Z_n}^{(k,j)}$.

\begin{corollary}\label{cor:cyclic-dimension}
    The dimension of $R_{\Z_n}^{(k,j)}$ is given by 
    \begin{equation}
        \sum_{i=0}^{n-1}\sum_{\ell=0}^i \binom{k+\ell-1}{\ell} \binom{j}{i-\ell}.
    \end{equation}
\end{corollary}

\begin{proof}
    This follows from Theorem~\ref{thm:cyclic-hilbert} specialized using equation~(\ref{eq:super-schur-specialization}).
\end{proof}

Finally, we determine a monomial basis for $R_{\Z_n}^{(k,j)}$.

\begin{theorem}
    A basis $B_{\Z_n}^{(k,j)}$ for the coinvariant ring $R_{\Z_n}^{(k,j)}$ is given by
    \begin{equation}\{x_{1}^{\alpha_1} \cdots x_{k}^{\alpha_k}\theta_{1}^{\beta_1} \cdots \theta_{j}^{\beta_j} \, | \, 0 \leq \alpha_1 + \cdots + \alpha_k + \beta_1 + \cdots + \beta_j \leq n-1\},\end{equation}
    where $\beta_i \in \{0,1\}$ for all $i \in \{1,\ldots,j\}$.
\end{theorem}

\begin{proof}
    It is not difficult to see that the invariant ideal $\langle\C[x_1,\ldots,x_k, \theta_1,\ldots, \theta_j]_+^{\Z_n}\rangle$ contains all monomials $x_1^{\alpha_1} \cdots x_k^{\alpha_k} \theta_1^{\beta_1} \cdots \theta_j^{\beta_j}$ where $\alpha_1+\cdots+\alpha_k+\beta_1+\cdots+\beta_j = n$ and $\beta_i \in \{0,1\}$ for all $i \in \{1,\ldots,j\}$. 
Furthermore, there are no non-constant invariants of lower total degree.
We can use these as straightening relations to show that the claimed basis is a spanning set.

To show that the spanning set is a basis, it suffices to show that its enumeration matches that in Corollary~\ref{cor:cyclic-dimension}. Let $i = \alpha_1 + \cdots + \alpha_k +\beta_1 + \cdots + \beta_j$ and use a stars and bars argument to complete the proof.
\end{proof}

\section*{Acknowledgements}

The author would like to thank Joseph Alfano, Fran\c{c}ois Bergeron, Sylvie Corteel, Christian Gaetz, Nicolle Gonz\'alez, Mark Haiman, Eric Jankowski, Brendon Rhoades, Vera Serganova, Josh Swanson, and Nolan Wallach for helpful conversations, sharing references, or feedback. 
The author was partially supported by the National Science Foundation Graduate Research Fellowship DGE-2146752.

\bibliographystyle{amsplain}
\bibliography{biblio}

\end{document}